\documentclass{amsart}
\usepackage{mathrsfs}

\usepackage{amsfonts}
\usepackage{mathrsfs}
\usepackage{amssymb, latexsym}
\usepackage{graphicx}
\usepackage{amsmath}
\usepackage[pdftex]{color} % for the colored block

\theoremstyle{plain}
\newtheorem{thm}{Theorem}[section]
\newtheorem{lemma}[thm]{Lemma}

\theoremstyle{definition}

\catcode`\@=11
\def\mequal{\mathrel{\mathpalette\@mvereq{\hbox{\sevenrm m}}}} 
\def\@mvereq#1#2{\lower.5\p@\vbox{\baselineskip\z@skip\lineskip1.5\p@
    \ialign{$\m@th#1\hfil##\hfil$\crcr#2\crcr=\crcr}}}
\def\partr#1#2{/\kern-.08333em/_{#1,#2}^{\phantom{.}}}
\def\invpartr#1#2{/\kern-.08333em/_{#1,#2}^{-1}} 
\def\hpartr#1#2{/\kern-.08333em/_{#1,#2}^{h}}
\def\Epartr#1#2{/\kern-.08333em/_{#1,#2}^{E}}

\def\newdot{{\kern.8pt\cdot\kern.8pt}}
\def\,{\relax\ifmmode\mskip\thinmuskip\else\thinspace\fi}
\def\{{\relax\ifmmode\lbrace\else $\lbrace$\fi}
\def\}{\relax\ifmmode\rbrace\else $\rbrace$\fi}

\font\sevenrm=cmr7

\newcommand\TT{\mathbb{T}}
\newcommand\ZZ{\mathbb{Z}}

\def\mathpal#1{\mathop{\mathchoice{\text{\rm #1}}%
   {\text{\rm #1}}{\text{\rm #1}}%
   {\text{\rm #1}}}\nolimits}

\def\div{\mathpal{div}}

\def\T{{\mathbb T^d}}

\begin{document}

\title[]{A stochastic variational approach to the viscous Camassa-Holm and Leray-alpha equations}

\author[A.B. Cruzeiro and G. Liu]{Ana Bela Cruzeiro (1) and  Guoping Liu (2)} \address{(1), (2) GFMUL and Dep. de Matem\'atica Instituto Superior T\'ecnico (UL) \hfill\break\indent
  Av. Rovisco Pais\hfill\break\indent
  1049-001 Lisboa, Portugal \hfill\break\indent
  (2)  AMSS, CAS \hfill\break\indent
  Zhongguancun East Road, No.55\hfill\break\indent
  100190 Beijing, P.R.China
  }
\email{(1) abcruz@math.tecnico.ulisboa.pt}
\email{(2) liuguoping002@163.com}

%\keywords{}

%%%%%%%%%%%%%%%%%%%%%%%%%%%%%%%%%%%%%%%%%%%%%%%%%%%%%%%%%%%%%%%%%%%%%%%%%%
%
%  Abstract, Keywords, AMS classification
%
%%%%%%%%%%%%%%%%%%%%%%%%%%%%%%%%%%%%%%%%%%%%%%%%%%%%%%%%%%%%%%%%%%%%%%%%%%

\begin{abstract}\noindent
We derive the ($d$-dimensional) periodic   incompressible and  viscous Camassa-Holm equation as well as the Leray-alpha equations via  stochastic variational principles. We discuss the existence of solution 
for these equations in the space $H^1$ using the probabilistic characterisation.  The underlying Lagrangian flows are diffusion processes living in the group of diffeomorphisms of the torus. We study in detail these diffusions.

\end{abstract}

\maketitle
\tableofcontents

%%%%%%%%%%%%%%%%%%%%%%%%%%%%%%%%%%%%%%%%%%%%%%%%%%%%%%%%%%%%%%%%%%%%%%%%%%%%
%
%  Actual Body of the Paper
%
%%%%%%%%%%%%%%%%%%%%%%%%%%%%%%%%%%%%%%%%%%%%%%%%%%%%%%%%%%%%%%%%%%%%%%%%%%%%

%%%%%%%%%%%%%%%%%%%%%%%%%%%%%%%%%%%%%%%%%%%%%%%%%%%%%%%%%%%%%%%%%%%%%%%%%%%%
\section{Introduction}\label{Section1}
\setcounter{equation}0

In the context of fluid dynamics the one dimensional Camassa-Holm equation
$$
\frac{\partial}{\partial t}(u-u^{\prime \prime})=
-3 u u^{\prime} +2u^{\prime}u^{\prime \prime}+uu^{\prime \prime \prime}
\leqno{(1.1)}
$$
was introduced in \cite{Camassa-Holm:93} to
describe the motion of unidirectional shallow water waves. The Lagrangian approach  consists in looking at the integral flows associated to the velocities $u$, namely the curves $g (t)(x)$ satisfying

$$\frac{\partial}{\partial t} g(t)(x)=u(t, g(t)(x)),~~~g(0)(x)=x. \leqno(1.2)$$
The Lagrangian flows $g(t) (\cdot )$ corresponding to the Camassa-Holm equation are geodesics with respect to the right-invariant
induced $H^1$ metric on a group of   (Sobolev) homeomorphisms over the circle.
This was proved in  \cite{Misiolek:98} and \cite{Kouranbaeva:99} and it corresponds to Arnold's (\cite{Arnold:66}) characterization of the Euler equation, now replacing the $L^2$ by
the right invariant $H^1$ norm in the Lagrangian.

Geodesics are minima of length and there is, indeed, a variational principle associated to the equations.
The Camassa-Holm Lagrangian flows $g(t)$, with $t\in [0,T]$,  can be characterized as critical paths for the action
functional

$$S[g]=\frac{1}{2}E\int_0^T ||\dot g(t)\circ g^{-1} (t)||_{H^1}^2 dt \leqno(1.3)$$
where $\dot g$ denotes the derivative in time of $g$.

These geodesic equations are  a special case of Lagrangian systems treated in Geometric Mechanics
via variational principles in  general Lie groups (\cite{Marsden-Ratiu:02}).

 In \cite{Shkoller:98}, and developing infinite-dimensional geometric methods as in \cite{Ebin-Marsden:70}, the well-posedness of the problem in
the space $H^s$ with $s>\frac{3}{2}$ was proved. Also in \cite{Constantin-Escher:98} the existence of solutions in $H^3$ was shown.
Camassa-Holm equation is also studied in higher dimensions in many works: we refer here  to \cite{Holm-Marsden:04} and
\cite{Shkoller:98}. 

Considering this equation in the viscous case, namely
\begin{eqnarray*}
\frac{\partial}{\partial t}(u- u^{\prime \prime})-\nu (u- u^{\prime \prime})^{\prime \prime}=
-3uu^{\prime}+2u^{\prime}u^{\prime\prime}+uu^{\prime\prime\prime}
\end{eqnarray*}
and, in  higher dimensions,

\begin{eqnarray*}
&&\frac{\partial}{\partial t}(u-\Delta u)-\nu\Delta(u-\Delta u)\\
&=&-u\cdot \nabla (u-\Delta u)- \div  u(u-\Delta u)+\sum_j
\nabla u^j\cdot \Delta u^j-\frac{1}{2}\nabla |u|^2,
\end{eqnarray*}
we lose the geodesic characterization, as we are no longer dealing with a conservative system. Nevertheless we still have a variational model, in a stochastic framework. The aim of this article is to formulate a stochastic variational principle for the  viscous incompressible Camassa-Holm equation and prove an existence result  of the critical (Lagrangian) stochastic process. The Lagrangian process provides a solution for this equation:
in our stochastic variational framework we replace deterministic Lagrangian paths  by  semimartingales and consider the classical Lagrangian evaluated on the drift of those semimartingales, this drift playing the role of  their (mean) time derivative. The critical paths
for the action are  diffusions whose drift satisfies Camassa-Holm equation. 

The viscous Camassa-Holm equations are also known as the Navier-Stokes-alpha  equations. In our work we consider $\alpha$ equal to one, for simplicity of notation
(this does not have any implication in the results). A very similar model, known as  the Leray-alpha equations, was introduced in \cite{Cheskidov-etal:05}. We study also  these
equations (again, with $\alpha =1$) from the variational point of view. They are obtained when considering the same action functional, but a different class of admissible variations.

The stochastic variational principle was derived in the case of the $L^2$ metric and for the Navier-Stokes equations, in the two-dimensional torus,
in \cite{Cipriano-Cruzeiro:07} and later generalized
to compact Riemannian manifolds in \cite{Arnaudon-Cruzeiro:12}.  
The same kind of  variational principles can be formulated on general
Lie groups: this is the content of reference \cite{Arnaudon-Chen-Cruzeiro:12} (c.f. also \cite{Chen-Cruzeiro-Ratiu:15}).

Viscous Camassa-Holm equations were introduced in \cite{Chen-etal:98}, \cite{Chen-etal:99} and \cite{Foias-etal:02} as a model of fluid  turbulence 
(c.f. also \cite{Foias-etal:01}). Several works are devoted to the existence of solutions for this equation in Sobolev spaces and under different boundary conditions via partial differential equation
methods. Just to cite some, we mention  \cite{Tan-Shen-Ding:07}, \cite{Lim:07} or \cite{BS:2008}.

Our study is probabilistic. It may be regarded from the perspective of stochastic control theory or stochastic geometric mechanics.

We study the  incompressible viscous Camassa-Holm equation with periodic boundary conditions in the space variable (i.e. in the $d$-dimensional torus
$\mathbb T^d $). In the next section we construct some Brownian motions living in the homeomorphisms group and show the existence of stochastic flows which are perturbations of  such Brownian motions by a time dependent drift $u\in L^2 ([0,T]; H^1 (\mathbb T^d ))$. In section 3 we derive a stochastic variational principle for the $H^1$ metric and in the following section we prove the existence of a weak solution for the  periodic incompressible viscous Camassa-Holm equation  under some restrictive assumption. For the Leray-alpha model and using a different class of admissible variations this assumption can be removed: this is done in the last paragraph.

\section{Stochastic processes on the homeomorphisms group of the torus}\label{Section2}
\setcounter{equation}0

2.1. The group of homeomorphisms of the torus.

\vskip 2mm
 
Let $\mathbb T^d$ be the $d$-dimensional flat torus. We denote by $G^s, s\geq 0$ the (infinite-dimensional) group of maps $g: \mathbb T^d \rightarrow 
\mathbb T^d$ such that  $g$ and $g^{-1}$ belong to the Sobolev space $H^s$ and such that $g$ keeps the volume measure  $d\theta$ 
invariant, namely $(g)_* (d\theta )=d\theta$. When $s>\frac{d}{2}+1$   Sobolev imbedding theorems imply that  the maps of $G^s$ 
are diffeomorphisms.
Also $G^s $ is a topological group for the composition of maps (not quite a Lie group because
left  composition  is not a smooth operation) and it is a smooth manifold (c.f. \cite{Ebin-Marsden:70}).

Let us denote by $e$ the identity of the group, $e (\theta )=\theta $. The tangent space at the identity  consists 
of  vector fields on $M$ which are $H^s $ regular and have zero divergence. This space is identified 
with the Lie algebra of the group.
On $G^s \equiv G^s (\T )$ we consider the $H^1$ metric: for $X,Y \in T_e G^s$,

$$<X ,Y >_{H^1} =\int_{\T } (X (\theta ).Y (\theta ))d\theta 
+\int_{\T } (\nabla X (\theta ). \nabla Y(\theta ))  d\theta.$$

One
extends this metric by right invariance to $G^s$, namely
if $X,Y \in T_g (G^s )$, the space of $H^s$ vector fields over $g$,

$$(<X,Y>_{H^1})_{g}=<X\circ g^{-1}, Y\circ g^{-1}>_{H^1}.$$

Note  that this metric  does not necessarily coincide with the one that defines the topology (this is an example of 
a \it weak Riemannian structure\rm). We refer to   \cite{Ebin-Marsden:70}   for a detailed study of the geometry of 
diffeomorphism groups.

 We want to construct a basis for the Lie algebra of the group $G^s $ endowed with the $H^1$ metric. Let ${\tilde \ZZ}^d$ be a subset
of $\mathbb Z^d$ 
where we identify  $k,l \in \mathbb Z^d$ through the equivalence relation  $k\simeq l$ iff $k+l =0$.

Consider, for each $k\neq 0$,
 an orthonormal basis $\{ \epsilon_k^1 ,~..., \epsilon_k^{d-1} \}$ of the space
$E_k =\{ x\in \mathbb R^d : k.x=0 \}$, where $k.x$ denotes scalar product. 
For example, when $d=2$ we take $\epsilon_k =\frac{1}{|k|}(-k_2 ,k_1 )$ when $k=(k_1 ,k_2 )$. For $d>2$ there is no canonical choice of basis. 
Write $\epsilon_{-k}=-\epsilon_k$.

Considering the usual identification of vector fields and functions, $u\rightarrow u(\theta ).\nabla$, a
 basis of the Lie algebra of $H^s$ divergence free vector fields on the $d$-dimensional torus can be defined as

$$\{ \lambda_k (s) (\epsilon_k^\alpha )\cos (k.\theta ), \lambda_k (s) (\epsilon_k^\alpha )\sin (k.\theta )\}_{k\in \tilde \ZZ^d , \alpha =1,..., d-1}$$
for some normalizing factors $\lambda_k (s)$, together with the canonical vector fields of $\mathbb R^d$.

In this work we typically denote by $k,l$ indeces in $\tilde \ZZ^d $, by $\alpha ,\beta =1, ..., d-1$ indices for the basis of of $E_k$ and $i,j =1, ..., d$
the $d$-dimensional  components of the torus.

\vskip 5mm

2.2. Brownian motions on $G^0$.

\vskip 2mm

Let
$\{x_{k}^{\alpha,1}(t,\omega ),
x_{k}^{\alpha,2}(t,\omega )$ with  
$k \in {\tilde \ZZ}^d , k\neq 0, \alpha =1,~..., d-1,  t\geq 0$,
  be a sequence of real valued independent standard Brownian motions defined 
on a given  filtered probability space $(\Omega,\mathcal{F},P)$, with increasing filtration  $\mathcal{F}_t$,   and let $y(t,\omega )$ be a $\TT^d$-valued Brownian motion with components which are also independent from the previous real-valued Brownian motions. We shall drop the probability space
 parameter in the notations.

 Define
$\alpha_k^2=(|k|^2+1)^{\frac{r}{2}}$ with $r\geq d+1$. Then
$$
x(t)(\theta)=\sum_{k\neq 0}\sum_{ \alpha} 
\frac{1}{\alpha_k}(\epsilon_k ^\alpha) [x_{k}^{\alpha,1}(t)\cos(k.\theta)+
x_{k}^{\alpha,2}(t)\sin(k.\theta)]+y(t)
\leqno{(2.1)}
$$
converges uniformly on $[0,T]\times \TT^d$.

The  canonical horizontal diffusion  corresponding to the Lie algebra valued process $x(t)$  is the solution of the  following Stratonovich stochastic differential equation 
with respect to the filtration $\mathcal{F}_t$, $t\in [0,T]$,
$$dg(t)=(\circ dx(t))(g(t)),\qquad
g(0)=e. \leqno{(2.2)}
$$
More explicitly, for $i=1,~..., d$,

$$
dg_i (t) (\theta )=\sum_{k\neq 0}\sum_{ \alpha} 
\frac{1}{\alpha_k}(\epsilon_k^\alpha )_i [\cos(k.g(t)(\theta ))\circ dx_{k}^{\alpha,1}(t)+\sin(k.g(t)(\theta ))\circ dx_{k}^{\alpha,2}(t)]+  dy_i (t).
$$
$g_i (0)(\theta )=\theta_i$

\begin{lemma}
Equation (2.2) can be written in the It\^o form as follows:
$$dg_i (t) (\theta )=\sum_{k\neq 0}\sum_{\alpha} 
\frac{1}{\alpha_k}(\epsilon_k^\alpha )_i [\cos(k.g(t)(\theta))dx_{k}^{\alpha,1}(t)+\sin(k.g(t)(\theta )) dx_{k}^{\alpha,2}(t)]+ dy_i (t)\leqno{(2.3)}$$
$g_i (0)(\theta )=\theta_i$.
\end{lemma}

\begin{proof}
We have,
\begin{eqnarray*}
&&d\cos(k.g(t)).dx_{k}^{\alpha,1}(t)\\
&=&-\sum_i k_i \sin(k.g(t))dg_i (t).dx_{k}^{\alpha,1}(t)\\
&=&-\sum_{i,\beta ,m \neq 0}  k_i \sin(k.g(t))
\frac{1}{\alpha_m}(\epsilon_m^\beta )_i [\cos(m.g(t))\circ dx_{m}^{\beta,1}(t)+\sin(m.g(t))\circ dx_{m}^{\beta,2}(t)].dx_{k}^{\alpha,1}(t)\\
&=&-\sum_{i} \frac{1}{\alpha_k}k_i (\epsilon_k^\alpha )_i \sin(k.g(t))\cos(k.g(t))dt
\end{eqnarray*}
and
\begin{eqnarray*}
&&d\sin(k.g(t)).dx_{k}^{\alpha,2}(t)\\
&=&\sum_i k_i \cos(k.g(t))dg_i (t).dx_{k}^{\alpha,2}(t)\\
&=&\sum_{i, \beta ,m\neq 0 }k_i \cos(k.g(t))
\frac{1}{\alpha_m}(\epsilon_m^\beta )_i [\cos(m.g(t))\circ dx_{m}^{\beta,1}(t)+\sin(m.g(t))\circ dx_{m}^{\beta,2}(t)].dx_{k}^{\alpha,2}(t)\\
&=&\sum_{i} \frac{1}{\alpha_k}k_i (\epsilon_k^\alpha )_i  \sin(k.g(t))\cos(k.g(t))dt
\end{eqnarray*}
The It\^o stochastic contraction is therefore equal to zero
and the conclusion follows.
\end{proof}
We prove the existence of the process $g(t)$ in the case where $r =d+3$.
For this we shall need the following lemma.
\begin{lemma}
Define 
$$V(\theta)=\sum_{k\neq 0}\frac{\sin^2(k\cdot \theta)}{|k|^{d+3}}.$$
Then there exists a constant $C_1>0$ such that for all $0<|\theta|<\frac{\pi}{4}$,
$$V(\theta)\leq C_1|\theta|^2\log\frac{1}{|\theta|}.$$
\end{lemma}
\begin{proof}

We have,
$$V(\theta)=\sum_{k\neq 0}\frac{\sin^2(k\cdot \theta)}{|k|^{d+3}}
\leq C_d\sum_{k\neq 0}\sum_{i=1}^d\frac{\sin^2(k_i\cdot \theta_i)}{|k|^{d+3}}$$
%&\leq&C_d\sum_{k\neq 0}\frac{1}{(|k|^2-k_i^2)^\frac{d}{2}}\sum_i\frac{\sin^2(k_i\cdot \theta_i)}{k_i^3}\\
%&\leq&C_1|\theta|^2\log\frac{1}{|\theta|}

Using the following  inequality of arithmetic and geometric means
$$\frac{\frac{k_i^2}{3}+\frac{k_i^2}{3}+\frac{k_i^2}{3}+(k_1^2+\cdots+\breve{k}_i^2+\cdots+k_d^2)}{4}\geq \sqrt[4]{(\frac{k_i^2}{3})^3(k_1^2+\cdots+\breve{k}_i^2+\cdots+k_d^2)},$$
i.e. $|k|^4\geq Ck_i^3(k_1^2+\cdots+\breve{k}_i^2+\cdots+k_d^2)^{\frac{1}{2}},$ we obtain $|k|^{d+3}=|k|^4|k|^{d-1}\geq Ck_i^3(k_1^2+\cdots+\breve{k}_i^2+\cdots+k_d^2)^{\frac{d}{2}},$ where the symbol $\breve{\cdot }$ represents omission of the corresponding term.

Since $\sum\limits_{k\neq 0}\frac{1}{(k_1^2+\cdots+\breve{k}_i^2+\cdots+k_d^2)^{\frac{d}{2}}}$ converges, we have
\begin{eqnarray*}
V(\theta)&\leq&C_d\sum_{i=1}^d\sum_{k\neq 0}\frac{\sin^2(k_i\cdot \theta_i)}{|k|^{d+3}}\\
&\leq&C_d\sum_{i=1}^{d}(\sum_{k\neq 0}\frac{\sin^2(k_i\cdot \theta_i)}{k_i^3})(\sum_{k\neq 0} \frac{1}{(k_1^2+\cdots+\breve{k}_i^2+\cdots+k_d^2)^{\frac{d}{2}}})\\
&\leq&\widetilde{C}_d\sum_{i=1}^{d}\sum_{k\neq 0}\frac{\sin^2(k_i\cdot \theta_i)}{k_i^3}\\
&\leq&C\sum_{i=1}^{d}|\theta_i|^2\log\frac{1}{|\theta_i|}
\end{eqnarray*}
where $\widetilde{C}_d=C_d\sum\limits_{k\neq 0} \frac{1}{(k_1^2+\cdots+\breve{k}_i^2+\cdots+k_d^2)^{\frac{d}{2}}}$,
and the last inequality comes from lemma 2.1 in \cite{Fang:02}.

Noticing the fact that the function $\xi \rightarrow \xi\log\frac{1}{\xi}$ is concave over $]0,1[$,
\begin{eqnarray*}
2\sum_{i=1}^{d}|\theta_i|^2\log\frac{1}{|\theta_i|}&=&\sum_{i=1}^{d}|\theta_i|^2\log\frac{1}{|\theta_i|^2}\\
&\leq&(\sum_{i=1}^{d}|\theta_i|^2)\log\frac{1}{(\sum_{i=1}^{d}|\theta_i|^2)}\\
&=&|\theta|^2\log\frac{1}{|\theta|^2}\\
&=&2|\theta|^2\log\frac{1}{|\theta|}\\
\end{eqnarray*}

Finally, we get the result $V(\theta)\leq C_1|\theta|^2\log\frac{1}{|\theta|}.$
\end{proof}
\vskip 3mm
\bf\noindent Theorem 2.3. \it
For $r =d+3$ the solution $g(t)$ of the stochastic differential equation $(2.2)$ exists and is a continuous process with values in the space 
of the homeomorphism group $G^0$.\rm

\begin{proof}
The proof follows essentially  the arguments in  \cite{Fang:02}. 
For each $\theta\in \TT^d$  consider $g^n(t)(\theta)$  the solution of the following s.d.e.:
$$
d\gamma^n(t)=\sum_{|k|\leq 2^n ,k\neq 0}\frac{1}{\alpha_k}
\sum_{\alpha}(\epsilon_k^\alpha )[\cos(k.\gamma^n(t))dx_{k}^{\alpha,1}(t)+
\sin(k.\gamma^n(t))dx_{k}^{\alpha,2}(t)]+dy(t)
\leqno{(2.4)}
$$
$$\gamma^n(0)=\theta.$$
Denote $\eta_i (t)=\frac{\gamma_i^n(t)-\gamma_i^{n+1}(t)}{4}, i=1,...,d$; then the It\^o's stochastic contraction is 
\begin{eqnarray*}
d\eta_i (t)\cdot d\eta_i (t)&=&(\frac{1}{4})^2\{\sum_{|k|=1}^{2^n}
\frac{1}{\alpha_k^2}\sum_{\alpha} [(\epsilon_k^\alpha )_i]^2
[(\cos k.\gamma^{n+1}(t)-\cos k.\gamma^n(t))^2\\
&&+(\sin k.\gamma^{n+1}(t)-\sin k.\gamma^n(t))^2]\\
&&+\sum_{|k|=2^n+1}^{2^{n+1}}\frac{1}{\alpha_k^2}\sum_{\alpha} 
[(\epsilon_k^\alpha )_i]^2 [\cos^2(k.\gamma^{n+1}(t))+\sin^2(k.\gamma^{n+1}(t))]\}\\
&\leq&(\frac{1}{4})^2 [(d-1)\sum_{|k|=1}^{2^n}
\frac{1}{\alpha_k^2}4 \sin^2(k\frac{\gamma^{n+1}(t)-\gamma^n(t)}{2})+
(d-1)\sum_{|k|=2^n+1}^{2^{n+1}}\frac{1}{\alpha_k^2}].
\end{eqnarray*}
Since $ [(\epsilon_k^\alpha )_i]^2\leq |\epsilon_k^\alpha|^2 \leq 1 $, the last inequality holds.
We have $\alpha_k^2=(|k|^2+1)^{\frac{d+3}{2}}\geq |k|^{d+3}.$   From lemma 2.2 we obtain
$$\sum_{|k|=1}^{2^n}\frac{1}{\alpha_k^2}4\sin^2(k.\frac{\gamma^{n+1}(t)-\gamma^n(t)}{2})\leq C_1|\eta(t)|^2\log\frac{1}{|\eta(t)|}$$ 
and
$$\sum_{|k|=2^n+1}^{2^{n+1}}\frac{1}{\alpha_k^2}\leq \sum_{|k|=2^n+1}^{2^{n+1}}\frac{1}{|k|^{d+3}}\leq C_22^{-n}.$$
where $C_2=\sum\limits_{k\neq 0} \frac{1}{|k|^{d+2}}$.

Using It\^o's formula, for $p\geq 1$, we obtain 
$$
d\eta_i^{2p}(t)=2p\eta_i^{2p-1}(t)\cdot d\eta_i(t)+p(2p-1)\eta_i^{2p-2}(t)d\eta_i(t)\cdot d\eta_i(t).
$$
It follows that
\begin{eqnarray*}
E^{\mathcal{F}_t}(\eta_i^{2p}(t+\varepsilon)-\eta_i^{2p}(t))\leq C_1p\int_t^{t+\varepsilon}E^{\mathcal{F}_t}(|\eta(s)|^{2p}\log\frac{1}{|\eta(s)|^{2p}}ds)+K_p2^{-n}\varepsilon,
\end{eqnarray*}
where $K_p=C_2p(2p-1)(\frac{\pi}{4})^{2p-2}$.

Denote $\varphi(t)=E(|\eta(t)|^2)$; we have 
\begin{eqnarray*}
\varphi^\prime(t)=\sum_i\frac{d}{dt}E(\eta_i^2(t))&\leq& C_1E(|\eta(t)|^{2}\log\frac{1}{|\eta(t)|^{2}})+C_22^{-n}\\
&\leq&C(\varphi(t)\log\frac{1}{\varphi(t)}+2^{-n}),
\end{eqnarray*}
where the last inequality comes from the fact that the function $\xi \rightarrow \xi\log\frac{1}{\xi}$ is concave over $]0,1[$. We can apply  Lemma 2.3 in \cite{Fang:02} to get 
$$\psi^{\prime}(t)\leq C\psi(t)\log\frac{1}{\psi(t)}, $$
where $\psi(t)=\varphi(t)+2^{-n}$. Now lemma 2.2 in \cite{Fang:02} gives 
$$\varphi(t)\leq \psi(t)\leq 2^{-n\delta(t)}$$
with $\delta(t)=e^{-Ct}.$   

Hence there exist a  constant $C>0$ such that,
 $$E(|g^n(t)(\theta)-g^{n+1}(t)(\theta)|^{2})\leq C2^{-n\delta(t)}.$$ 
By the martingale maximal inequality,
$$E(\sup\limits_{0\leq t\leq T}|g^n(t)(\theta)-g^{n+1}(t)(\theta)|^{2})\leq C2^{-n\delta(T)}.$$   
Using Borel-Cantelli we deduce that\\ 
$$g(t)(\theta)=\lim\limits_n g^n(t)(\theta)$$ exists uniformly in $t\in [0,T]$.

Following \cite{Fang:02}, one can show that $g(t)$ satisfies Eq.$(2.2)$ and that $g(t)$ is the unique solution of this equation.\\    
    
We consider now  the regularity of $g(t)$. Using the same computation as before, we have the following estimate  
$$E(\sup\limits_{0\leq t\leq T}|g^n(t)(\theta)-g^n(t)(\theta^{\prime})|^{2})\leq C|\theta-\theta^{\prime}|^{2\delta(T)}.$$
From the triangle inequality, we obtain
$$E(|g^n(t)(\theta)-g^{n+1}(t)(\theta^\prime)|^{2})\leq C[|\theta-\theta^\prime|^{2\delta(t)}+2^{-n\delta(t)}].$$ 
Then we apply the Kolmogorov modification theorem, so that almost surely $g_t^n\rightarrow g_t$ uniformly over $\T $. Since $\delta(t)\rightarrow 1$ as $t\rightarrow 0$, for any $0<\delta<1$ , there exists $t_0>0$ such that 
$$
|g(t)(\theta)-g(t)(\theta^{\prime})|\leq C_{\delta,t}|\theta-\theta^{\prime}|^\delta,~~~t\leq t_0.
\leqno{(2.5)}
$$
Denote $g_{t,x}^n$ the solution of  s.d.e. (2.4). Then for $s\leq t_0$, we have (see \cite{Ikeda-Watanabe:81} )
$$g^n_{t_0+s,x}(\theta)=g^n_{s,x^{t_0}}(g^n_{t_0,x}(\theta)),$$
where $t\rightarrow x^{t_0}(t)=x(t+t_0)-x(t_0)$ is again a Brownian motion.
For $t>t_0$, using the flow property and letting $n\rightarrow +\infty$
$$g_{t_0+s,x}(\theta)=g_{s,x^{t_0}}(g_{t_0,x}(\theta));$$
 together with (2.5) and writing $m=[\frac{t}{t_0}]$  the integral part of $\frac{t}{t_0}$, we obtain 
$$
|g_{t,x}(\theta)-g_{t,x}(\theta^{\prime})|\leq C_{\delta,t}|\theta-\theta^{\prime}|^{\delta^{m+1}}.
\leqno{(2.6)}
$$
We have $m+1\leq \frac{t}{t_0}+1\leq 2\frac{t}{t_0}$. Denote $c_0=\frac{2}{t_0}\log\frac{1}{\delta}$, then we deduce 
$$|g_{t,x}(\theta)-g_{t,x}(\theta^{\prime})|\leq C_t|\theta-\theta^{\prime}|^{e^{-c_0t}}$$
from (2.6).
Using this inequality, we can show that $g(t)(\cdot)$ is $\alpha$-H\"older continuous $(0<\alpha<e^{-c_0t})$.

Following \cite{Fang:02} and \cite{Malliavin:99}, the flow property can be used to prove that the stochastic process $g(t)$ lives in the space of homeomorphisms on $\TT^d$.
\end{proof}
\vskip 3mm
 
 Let us describe the infinitesimal generator of the stochastic process $g(t):\TT^d \rightarrow \TT^d$ defined on very special functionals, namely
 functionals defined $\theta$-pointwise.
 
 \vskip 2mm
\bf Definition. \rm
Let $f$ be a $C^2$  function   defined on $\TT^d$. On a functional $F(g)(\theta)=f(g(\theta)), \theta\in \TT^d$ ,
the infinitesimal generator of the process $g(t):\TT^d \rightarrow \TT^d$ is defined by 
\begin{eqnarray*}
L(F)(\theta)=\lim_{t\rightarrow 0} \frac{1}{t}E((g(t))^{\ast}f(\theta)-f(\theta)),
\end{eqnarray*}
where $(g(t))^{\ast}f(\theta)=f(g(t)(\theta))$.
\rm

\vskip 3mm
\bf\noindent Theorem 2.4. \it
Let $\mathcal{L}$ be the infinitesimal generator of the stochastic process $g(t)$ defined in (2.2). Then there exist strictly positive constants $c_1, ..., c_d$ such that, for $F(g)(\theta)=f(g(\theta))$,
$$\mathcal{L}(F (g))(\theta )=\sum_{i=1}^d c_i \partial_{i i}^2 f (g(\theta )), ~~~~~~~f\in C^2(\TT^d ).$$
If $d=2$ we have $c_1 =c_2 =c$ and the infinitesimal generator reduces to the Laplacian operator multiplied by $c$.

\begin{proof}
Let us denote $A_k^\alpha =(\epsilon_k^\alpha ) \cos(k.\theta )$ and
$B_k^\alpha =(\epsilon_k^\alpha ) \sin(k.\theta )$.
Using It\^o's formula we have,
\begin{eqnarray*}
df(g(t))&=&\sum_{k\neq 0}\sum_{\alpha}\frac{1}{\alpha_k} 
[(A_k^\alpha f) (g(t)) \circ dx_k^{\alpha,1} (t)+
(B_k^\alpha f) (g(t))\circ dx_{k}^{\alpha, 2} (t)]\\
&&+\sum_i \partial_i f(g(t)) \circ dy^i (t)\\
&=&\sum_{k\neq 0}\sum_{\alpha}\frac{1}{\alpha_k} 
[(A_k^\alpha f) (g(t))  dx_k^{\alpha,1} (t)+
(B_k^\alpha f) (g(t)) dx_{k}^{\alpha, 2} (t)]\\
&&+\sum_i \partial_i f(g(t))  dy^i (t)\\
&&+\frac{1}{2}\sum_{k\neq 0}\sum_{\alpha}\frac{1}{\alpha_k^2} 
[(A_k^\alpha(A_k^\alpha f)) (g(t))+(B_k^\alpha(B_k^\alpha f)) (g(t)) ]dt\\
&&+\frac{1}{2}\sum_i \partial_{ii}^2f(g(t))dt
\end{eqnarray*}
We compute the  It\^o stochastic contractions. We have

$$A_k^\alpha (A_k^\alpha )f (\theta )=\sum_{i,j} (A_k^\alpha )_i \partial_i (A_k^\alpha )_j \partial_j f (\theta )+\sum_{i,j}(A_k^\alpha )_i (\theta ) (A_k^\alpha )_j (\theta )\partial_{ij}^2 f(\theta )$$
and
$$B_k^\alpha (B_k^\alpha )f (\theta )=\sum_{i,j} (B_k^\alpha )_i \partial_i (B_k^\alpha )_j \partial_j f (\theta )+\sum_{i,j}(B_k^\alpha )_i (\theta ) (B_k^\alpha )_j (\theta )\partial_{ij}^2 f(\theta )$$

Let us consider the terms with  second order derivatives of $f$. When $i\neq j$ they are

$$\frac{1}{2}\sum_{k\neq 0,\alpha} \frac{1}{\alpha_k^2}\sum_{i\neq j} ((\epsilon_k^\alpha )_i )((\epsilon_k^\alpha )_j ) [\cos^2 (k.g(t)(\theta) )+\sin^2 (k.g(t)(\theta) )]\partial_{ij}^2f (g(t)(\theta) ) dt$$
$$= \frac{1}{2}\sum_{k\neq 0}\sum_{\alpha} \frac{1}{\alpha_k^2}\sum_{i\neq j} ((\epsilon_k^\alpha )_i )((\epsilon_k^\alpha )_j )\partial_{ij}^2f (g(t)(\theta) ) dt.$$
As we sum in all k (and -k), this term is zero.

For $i=j$ we have

$$\frac{1}{2}\sum_{k\neq 0}\sum_{\alpha} \frac{1}{\alpha_k^2}\sum_{i} [(\epsilon_k^\alpha )_i ]^2  [\cos^2 (k.g(t)(\theta ))+\sin^2 (k.g(t)(\theta) )]\partial_{ii}^2f (g(t)(\theta )) dt$$
$$=\frac{1}{2}\sum_{k\neq 0}\sum_{\alpha} \frac{1}{\alpha_k^2}\sum_{i} [(\epsilon_k^\alpha )_i ]^2 \partial_{ii}^2f (g(t)(\theta )) dt. $$
Define
$$a_i =\frac{1}{2}\sum_{k\neq 0}\sum_{\alpha}\frac{1}{\alpha_k^2}[(\epsilon_k^\alpha )_i ]^2$$

We now consider the terms involving first derivatives of $f$. They are equal to

$$\frac{1}{2}\sum_{k\neq 0,\alpha} \frac{1}{\alpha_k^2}\sum_{i\neq j} ((\epsilon_k^\alpha )_i )((\epsilon_k^\alpha )_j )[-k_i cos (k.g(t)(\theta) )\sin (k.g(t)(\theta) ) $$
$$+ k_i \sin (k.g(t)(\theta) )\cos (k.g(t)(\theta) ) ]\partial_j f (g(t)(\theta) )dt=0.$$
Therefore these terms  vanish.

We conclude that the infinitesimal generator is given by
$\mathcal{L}(F (g))(\theta )=\sum\limits_{i=1}^d c_i \partial^2_{i i} f (g(\theta ))$, $f\in C^2(\TT^d )$,
with $c_i =a_i +\frac{1}{2}$.

When $d=2$ we have only one $\alpha$,   $(\epsilon_k )_1 =-\frac{k_2}{|k|^2}$ and  
$(\epsilon_k )_2 =\frac{k_1}{|k|^2}$ (for $k=(k_1 ,k_2 )$), from which the result follows, in this case
with $c= \frac{1}{2}(\sum\limits_{k\neq 0}  \frac{ k_1^2 }{\alpha_k^2}+1)$.
\end{proof}
\vskip 3mm

\bf Remark 2.5. \rm
\vskip 1mm
1) We can consider only $x(t)=y(t)$. In this case the stochastic process is the standard Brownian motion and the generator of the process the Laplacian.
\vskip 1mm
2)
We can obtain the same result (up to a modification of the constants in front of the second order differential operator) by considering a Brownian motion defined by a finite sum, namely
$$x(t)(\theta)=\sum_{|k|\leq N, k\neq 0}\sum_{\alpha } \frac{1}{\alpha_k}(\epsilon_k^\alpha )[x_{k}^{\alpha,1}(t)\cos(k.\theta)+x_{k}^{\alpha,2}(t)\sin(k.\theta)]+y(t)
$$
Then the  corresponding Lagrangian flows are diffeomorphisms by well known results (c.f. \cite{Kunita:90}).

\rm
\vskip 5mm
Let us consider a time-dependent vector field  $u: [0,T]\rightarrow H^1 (\TT^d )$, $u\in L^2 ([0,T];H^1 )$, with $\div u (t,\cdot )=0$ for all $t$. We can associate to $u(t)$ the following $\mathcal{F}_t$ stochastic differential equation:

$$
dg^u_i (t) (\theta )=u_i (t, g^u (t))dt +\sum_{k\neq 0, \alpha} 
\frac{1}{\alpha_k}(\epsilon_k^\alpha )_i [\cos(k.g^u (t)(\theta ))\sqrt{\frac{\nu}{c_i}}\circ  dx_{k}^{\alpha,1}(t)$$
$$+\sin(k.g^u (t)(\theta ))\sqrt{\frac{\nu}{c_i}}\circ dx_{k}^{\alpha,2}(t)] +dy(t) ,\leqno{(2.7)}
$$
with $g^u_i (0)(\theta )=\theta _i$, where $c_i$ denote the constants defined in Theorem 2.4., $t\in [0,T]$. When $c_i =c ~~~~ \forall i$ it reads,

$$dg^u(t)=(u(t)dt+\sqrt{\frac{\nu}{c}}\circ dx(t))(g^u(t)),~~~~g^u(0)=e.$$

This equation can be also written with the Stratonovich differentials replaced by It\^o ones, as the It\^o contraction vanishes (lemma 2.1).

\vskip 3mm
The rest of this section is devoted to show that  this equation defines a stochastic flow.

\vskip 3mm
\bf\noindent Theorem 2.6. \it
Let $u$ belong to the space $L^2 ([0,T];H^1 (\T ))$ with $\div u (t,\cdot )=0$ for all $t$. Then there exists a stochastic flow $g^u$
which is a solution of the stochastic differential equation  (2.7).
\rm
\begin{proof}

This result can be found in \cite{Fang-Li-Luo:11}. There the authors assume that the time dependent vector field takes values in $L^2 ([0,T];H^q (\T ))$
for $q>2$ but their proof still holds for $q=2$. They actually prove the existence of a strong solution (c.f. also \cite{Cipriano-Cruzeiro:07} for a weak solution). Concerning the flow property, the inverse of $g_t^u$ is well defined and satisfies the sde

$$
dX_i (t) (\theta )=-u_i (t, X (t))dt +\sum_{k\neq 0, \alpha} 
\frac{1}{\alpha_k}(\epsilon_k^\alpha )_i [\cos(k.X(t)(\theta ))\sqrt{\frac{\nu}{c_i}}\circ  d\hat x_{k}^{\alpha,1}(t)$$
$$+\sin(k.X (t)(\theta ))\sqrt{\frac{\nu}{c_i}}\circ d\hat x_{k}^{\alpha,2}(t)]$$
where $\hat x_{k}^{\alpha,i}(t)= x_{k}^{\alpha,i}(T)-x_{k}^{\alpha,i}(T-t)$ are time-reversed Brownian motions.

\end{proof}
\vskip 1mm
\bf\noindent Corollary 2.7.
\it Suppose that $u$ satisfies the hypothesis  of  Theorem 2.6. and let $g^u(t)$ be the solution of Eq.(2.7). Then the infinitesimal generator of this process, when computed at functionals of the form $F(g)(\theta)=f(g(\theta))$, is given by
$${\mathcal{L}}^u (F (g))(\theta )=u(t,g(\theta )).\nabla  f (g(\theta ))+\sum_{i=1}^d c_i \partial_{i i}^2 f (g(\theta )), ~~~~~\forall f\in C^2(\T ).$$

\rm
\vskip 7mm

\section{Stochastic variational principle for the Camassa-Holm equation}\label{Section3}
\setcounter{equation}0

Let $\mathcal S$ denote the set of continuous semimartingales taking values in the measure-preserving  homeomorphism group of $\TT^d$.

 We consider $D_t$ the generalized time derivative of semimartingales. If $F$ is a smooth function on $\TT^d$
 and $\xi \in \mathcal S$,
 
$$D_t F(\xi_t )=\lim_{\epsilon \rightarrow 0} \frac{1}{\epsilon}[E_t F(\xi (t+\epsilon ))-F(\xi (t))],$$ 
 where $E_t$ denotes conditional expectation with respect to $\mathcal F_t$. When $\xi =g^u \in \mathcal S_0$,
 we have
 $$D_t g^u_t =u(t,g^u_t ).$$

For functionals defined on $\mathcal S$ we consider the following variations: to a vector field $v\in C^1 ([0,T]\times C^\infty (\TT^d ))$ with $v(0,\cdot )
v(T,\cdot )=0$
and $\div v(t,\cdot )=0$ for all $t$,
we associate $e_t (v)$  the solution of the ordinary differential equation
 
 $$\frac{d}{dt} e_t^v (\theta )=\dot v (t, e_t^v (\theta )), ~~~~~~~ e_0^v (\theta )=\theta .$$
 The  admissible variations of $\xi \in \mathcal S$ will be the set of $e_t^v (\xi_t )$, 
(which are still semimartingales in $\mathcal S$).

 \vskip 5mm
 
\bf Definition. \rm  Let $J $ be a real-valued  functional defined on $\mathcal S$. Consider left and right derivatives of $J$ at 
a semimartingale $\xi$ along directions $e^v_{\cdot} $, $v\in C^1 ([0,T]\times C^\infty (\T ))$ with $v(0,\cdot )=v(T,\cdot )=0$, namely

$$(D_l )_{e^v_{\cdot}}J[\xi ]=\frac{d}{d\epsilon}_{|_{\epsilon =0}} J[e^{ \epsilon v}_{\cdot} \circ\xi (\cdot )]$$

$$(D_r )_{e^v_{\cdot}}J[\xi ]=\frac{d}{d\epsilon}_{|_{\epsilon =0}} J[\xi (\cdot )\circ e^{ \epsilon v}_{\cdot} ]$$
A semimartingale $\xi$ is said to be critical for $J$ if
$$(D_l)_{e^v_{\cdot}}J[\xi ]=(D_r )_{e^v_{\cdot}}J[\xi ]=0$$
for all $v$ as above.

We can now formulate our stochastic variational principle. 
If $g_t^u$ is a solution of (2.7)  for some $ u \in L^2 ([0,T]; H^1 )$,
$\div u (t,\cdot )=0$, then its drift is of the form
$D_tg_t^u = u(t, g_t^u )$. For such a process define the following stochastic action:

$$A[g^u]=\frac{1}{2}E\int_0^T ||(D_t g^u_t )\circ (g^u_t )^{-1} (\cdot )||^2_{H^1} dt $$
$$= \frac{1}{2}\int_0^T||u(t,\cdot )||^2_{H^1} dt
 \leqno{(3.2)}$$
 
We have the following 
\vskip 3mm
\bf Theorem 3.1. \it Let $g_t^u $ be solution of (2.7) for some $ u \in L^2 ([0,T]; H^1(\T ))$ with
$\div u (t,\cdot )=0$ for all $t\in [0,T]$.  Then
$g^u$ is critical for the action functional $A$ if and only if there exists $p\in L^2 ([0,T]; H^1 )$ such that the vector field $u(t)$ satisfies in the weak ($L^2$) sense the viscous Camassa-Holm
equation

\begin{eqnarray*}
\frac{\partial}{\partial t}(u- {\mathcal L} u)-\nu {\mathcal L} (u- {\mathcal L} u)=
-u \cdot \nabla (u- {\mathcal L}  u)+
\sum_j \nabla  u^j \cdot  {\mathcal L} u^j -\nabla p
\end{eqnarray*}
with $u(T, \cdot )=u_T (\cdot )$, $\div u(t,\cdot )=0 ~\forall t$.

\rm
\vskip 3mm

When ${\mathcal L}=\Delta$, it is the usual Camassa-Holm equation.
Since the operator ${\mathcal L}$ has similar properties to the standard Laplacian operator $\Delta$, the proof for ${\mathcal L} $ is analogous and we shall write
it for $\Delta$.

Recall also that ${\mathcal L}$ reduces to the  Laplacian when $d=2$ and that, by choosing $x(t)=y(t)$ we can always consider the Laplacian case for $d>2$.
\vskip 3mm
\it Proof of Theorem 3.1. \rm
As the metric is right-invariant we only need to consider left derivatives.
Let $\varepsilon >0$. Since $e_0^v (\theta)=\theta$,
we have 
\begin{eqnarray*}
e_t^{\varepsilon v}=e+\varepsilon\int_0^t  \dot{v}(s,e_s^{\varepsilon v} )ds
\end{eqnarray*}
and
\begin{eqnarray*}
\frac{d}{d\varepsilon}|_{\varepsilon=0}e_t^{\varepsilon v}=\int_0^t  \dot{v}(s,\theta)ds=v(t,\theta).
\end{eqnarray*}
We denote  $e_{\varepsilon}(t,\theta)=e_t^{\varepsilon v}(\theta)$ and $g_t^u =g_t$. We have,
\begin{eqnarray*}
\frac{d}{d\varepsilon}|_{\varepsilon=0}A[e_{\varepsilon}\circ g]&=&
\sum_{i}E\int\int<\frac{d}{d\varepsilon}|_{\varepsilon=0}\{D[e_\varepsilon^i(t,g_t)]\circ (e_\varepsilon(g_t))^{-1}\},u^i(t,\theta)>dtd\theta\\
&&+\sum_{i,j}E\int\int<\frac{d}{d\varepsilon}|_{\varepsilon=0}\partial_{j}\{D[e_\varepsilon^i(t,g_t)]\circ (e_\varepsilon(g_t))^{-1}\},\partial_{j}u^i(t,\theta)>dtd\theta\\
&\triangleq&A_1+A_2
\end{eqnarray*}
Since 
\begin{eqnarray*}
D[e_\varepsilon^i(t,g_t)]=\partial_t e_\varepsilon^i(t,g_t)+(u\cdot\nabla e_\varepsilon^i)(t,g_t)+\nu\Delta e_\varepsilon^i(t, g_t)
\end{eqnarray*}
and
\begin{eqnarray*}
g_t\circ (e_\varepsilon(g_t))^{-1}=g_t\circ g_t^{-1}\circ e_\varepsilon^{-1}=e_\varepsilon^{-1},
\end{eqnarray*}
we have
\begin{eqnarray*}
&&\frac{d}{d\varepsilon}|_{\varepsilon=0}\{D[e_\varepsilon^i(t,g_t)]\circ (e_\varepsilon(g_t))^{-1}\}\\
&=&\frac{d}{d\varepsilon}|_{\varepsilon=0}\{\partial_t e_\varepsilon^i(t,e_\varepsilon^{-1})+\sum_{l}u^l(t,e_\varepsilon^{-1})\partial_l e_\varepsilon^i(t,e_\varepsilon^{-1})+\nu\Delta e_\varepsilon^i(t, e_\varepsilon^{-1})\}\\
&=&\partial_t v^i(t,\theta)+\sum_l [u^l(t,\theta)\partial_lv^i(t,\theta)-\partial_l u^i(t,\theta)v^l(t,\theta)]+\nu\Delta v^i(t,\theta)
\end{eqnarray*}
and
\begin{eqnarray*}
&&\frac{d}{d\varepsilon}|_{\varepsilon=0}\partial_j\{D[e_\varepsilon^i(t,g_t)]\circ (e_\varepsilon(g_t))^{-1}\}\\
&=&\partial_j\frac{d}{d\varepsilon}|_{\varepsilon=0}\{D[e_\varepsilon^i(t,g_t)]\circ (e_\varepsilon(g_t))^{-1}\}\\
&=&\partial_j\{\partial_t v^i(t,\theta)+\sum_l [u^l(t,\theta)\partial_lv^i(t,\theta)-\partial_l u^i(t,\theta)v^l(t,\theta)]+\nu\Delta v^i(t,\theta)\}\\
&=&\partial_t\partial_j v^i+\sum_l[\partial_ju^l\partial_lv^i+u^l\partial^2_{jl}v^i-\partial^2_{jl}u^iv^l-\partial_l u^i\partial_jv^l]+\nu\Delta\partial_jv^i.
\end{eqnarray*}
Then
\begin{eqnarray*}
A_1 &=&
\sum_{i}E\int\int<\frac{d}{d\varepsilon}|_{\varepsilon=0}\{D[e_\varepsilon^i(t,g_t)]\circ (e_\varepsilon(g_t))^{-1}\},u^i(t,\theta)>dtd\theta\\
&=&\sum_{i}E\int\int<\partial_t v^i(t,\theta)+\sum_l [u^l(t,\theta)\partial_lv^i(t,\theta)-\partial_l u^i(t,\theta)v^l(t,\theta)]+\nu\Delta v^i(t,\theta)
,u^i(t,\theta)>dtd\theta\\
&=&\sum_i[-\int\int\partial_tu^iv^idtd\theta+\nu\int\int\Delta u^iv^idtd\theta -\int\int(\operatorname{div} u)u^iv^idtd\theta\\
&&-\int\int(u\cdot\nabla)u^iv^idtd\theta-\frac{1}{2}\int\int\partial_i |u|^2v^idtd\theta]\\
&=&\int\int<\partial_tu,v>dtd\theta+\int\int<\nu\Delta u,v>dtd\theta -\int\int<(\operatorname{div} u)u,v>dtd\theta\\
&&-\int\int<(u\cdot\nabla)u,v>dtd\theta-\int\int<\frac{1}{2}\nabla |u|^2,v>dtd\theta,
\end{eqnarray*}
\begin{eqnarray*}
A_2 &=& \sum_{i,j}E\int\int<\frac{d}{d\varepsilon}|_{\varepsilon=0}\partial_{j}\{D[e_\varepsilon^i(t,g_t)]\circ (e_\varepsilon(g_t))^{-1}\},\partial_{j}u^i>dtd\theta\\
&=&\sum_{i,j}E\int\int<\partial_t\partial_j v^i+\sum_l[\partial_ju^l\partial_lv^i+u^l\partial^2_{jl}v^i-\partial^2_{jl}u^iv^l-\partial_l u^i\partial_jv^l]+\nu\Delta\partial_jv^i,\partial_j u^i>dtd\theta\\
&=&\sum_{i,j}[\int\int\partial_t\Delta u^i v^idtd\theta-\nu\int\int\Delta^2u^i v^i dtd\theta +\int\int \operatorname{div} u ~\Delta u^i v^i dtd\theta\\
&&+\int\int(u\cdot\nabla)\Delta u^i v^idtd\theta+\int\int\partial_iu^j\Delta u^j v^i dtd\theta]\\
&=&\int\int<\partial_t\Delta u, v>dtd\theta-\int\int<\nu\Delta^2u, v> dtd\theta +\int\int <\operatorname{div} u ~\Delta u,v> dtd\theta\\
&&+\int\int<(u\cdot\nabla)\Delta u, v>dtd\theta+\sum_j\int\int<\nabla u^j\Delta u^j, v> dtd\theta,
\end{eqnarray*}

Therefore, using the condition $\operatorname{div} u=0$, 
\begin{eqnarray*}
\frac{d}{d\varepsilon}|_{\varepsilon=0}A[e_{\varepsilon}\circ g^u]=0
\end{eqnarray*}
is equivalent to the equation
\begin{eqnarray*}
&&\frac{\partial}{\partial t}(u-\Delta u)-\nu\Delta(u-\Delta u)\\
&=&-u\cdot \nabla (u-\Delta u)+\sum_j
\nabla u^j\cdot \Delta u^j-\nabla p,
\end{eqnarray*}
satisfied in the weak sense (since $v$ is arbitrary).
\vskip 7mm

\bf Remark 3.2. \rm 
\vskip 1mm

 The proof of Theorem 3.1 can be also found, in a more general  group-theoretical framework, in \cite{Arnaudon-Chen-Cruzeiro:12}, where it was written for the case of the two-dimensional torus.

\section{Existence of a critical diffusion}\label{Section4}
\setcounter{equation}0

In this paragraph we discuss the existence of a critical diffusion, whose drift, a posteriori, will be a $H^1$ solution of the Camassa-Holm equation.

Recall that $\mathcal S$ is  the set of continuous semimartingales taking values in the measure-preserving  homeomorphism group of the torus. ${\mathcal S}_0$ will denote the subset of $\mathcal S$ consisting of diffusions $g^u$ that verify equation Eq.(2.7) for some drift $u\in L^2 ([0,T]; H^1 (\TT^d  ))$ with $\div u(t,\cdot )=0$
for all $t$. 
Let us consider the set of semimartingales of the form $g(t)=\eta (t, g_t^u (\theta) )$, with $g^u \in {\mathcal S}_0$, $\eta$ smooth in both variables and a measure-preserving diffeomorphism in $\theta$ with $\eta_t^{-1}$ also smooth. Denote this set by ${\mathcal S}_1$. We have ${\mathcal S}_0 \subset {\mathcal S}_1 \subset {\mathcal S}$.

Notice that, if $g_t^u$ is a semimartingale in ${\mathcal S}_0$ then all variations $e_t^v (g_t^u )$, $v$ as above, belong to  ${\mathcal S}_1$
and have a drift in $ L^2 ([0,T]; H^1)$.

Fix a vector field $z \in L^2 ([0,T];L^2 )$ and a  constant $c>0$. The action  functional,  defined   on the set of semimartingales in ${\mathcal S}_1$ 
for which the corresponding drifts satisfy the condition
$$\int \int  <u(t,\theta ), z(t,\theta > dt d~\theta \geq c,$$
 is bounded below.  Let $\alpha$ be its infimum.
Suppose that this infimum belongs to ${\mathcal S}_0$.  We consider $g^m(t)$ a minimizing sequence. If the minimum is attained in 
${\mathcal S}_0$ we can assume that
$g^m (t)=g^{u_m}(t)$, with $u_m \in L^2 ([0,T];H^1 )$.

We have the convergence $A[g^m(\cdot)]\rightarrow \alpha$ as $n\rightarrow \infty$. The sequence 
$A[g^m(\cdot)]=\|u_m\|^2_{L^2([0,T];H^1)}$
is bounded, therefore there exists a subsequence $u_{m_j}$ of $u_m$ that converges with respect to the weak topology, more precisely 
there exists $u\in L^2([0,T];H^1)$ such that
$$u_{m_j}\rightarrow u, ~~~~\hbox{weakly}~ in~ L^2([0,T]; H^1).$$
The limit function $u$ satisfies the assumptions of Theorem 2.6. Then we can construct a stochastic process $g^u(t)$ in $\mathcal{S}_0$ 
as solution of the stochastic differential equation (2.7). Since the norm is weakly lower semi-continuous, we have
$$A[g^u(\cdot)]\leq \lim\limits_{j\rightarrow\infty} A[g^{m_j}(\cdot)],$$
we deduce that $A[g^u(\cdot)]=\alpha$ and $g^u(t)$ is a minimum.

The curve of vector fields $u(t,\cdot )$
satisfies the incompressible viscous Camassa-Holm equation and it satisfies the condition
$$\int \int <u(t,\theta ), z(t,\theta > dt d\theta \geq c,$$
which is preserved by weak limits.

We have therefore proved the following result,
\vskip 3mm
\bf Theorem 4.1. \it  Let $z \in L^2 ([0,T];L^2 )$ be a vector field and $c$ a positive constant.
There exists a semimartingale $g(t)$ in the class ${\mathcal S}_1$
which realizes the minimum of the action functional $A$.
If the minimum belongs to ${\mathcal S}_0$ then the corresponding drift
$u(t, \cdot )$ satisfies the incompressible viscous Camassa-Holm equation in the weak $L^2$ sense.
Moreover  $\int \int <u(t,\theta ), z(t,\theta > dt d\theta \geq c$.
\rm
\vskip 3mm

\section{The Leray-alpha  equations}\label{Section5}
\setcounter{equation}0

The assumption that  the minimum of the action belongs to the space ${\mathcal S}_0$ considered in  last paragraph, is, of course, a priori quite strong.
In this section we consider the Leray-alpha equations (with $\alpha =1$) and we work with different variations that preserve the class of semimartingales ${\mathcal S}_0$. This allows us to remove this assumption and obtain a stronger result for this model.

For simplicity we assume in this section that $c_i =c~~~~\forall i$.

The incompressible  Leray-alpha equation, with $\alpha =1$, is

\begin{eqnarray*}
\frac{\partial}{\partial t}(u-\Delta u)-\nu\Delta(u-\Delta u)=-u\cdot \nabla(u-\Delta u)-\nabla p
\end{eqnarray*} 
with $\div u(t,\cdot )=0$ for all $t\in [0,T]$.

We consider the   same action functional as before, namely

$$A(g^u)=\frac{1}{2}E\int_0^T \|(D_t g_t^u)\circ(g_t^u)^{-1}(\cdot)\|_{H^1}^2dt.$$
The admissible variations to be of the form
$g_t^{\varepsilon}$ satisfying the following equation 
$$
dg_t^{\varepsilon}=(\sqrt{\frac{\nu}{c}}\nabla e_t^{\epsilon v}dx(t))(g_t^{\varepsilon})+[\partial_t e_t^{\epsilon v}+(u\cdot\nabla)e_t^{\epsilon v}+\nu\Delta e_t^{\epsilon v}](g_t^{\varepsilon})dt, \leqno{(5.1)}
$$
In particular we work only in the class ${\mathcal S}_0$. Considering the notion of criticality with respect to this new class of admissible variations, the following result holds

\vskip 3mm
\bf Theorem 5.1. \it Let $g_t^u $ be solution of (2.7) for some $ u \in L^2 ([0,T]; H^1(\T ))$ with
$\div u (t,\cdot )=0$ for all $t\in [0,T]$.
 Then
$g^u$ is critical for the action functional $A$ and with respect to the admissible class of variations (5.1) if and only if there exists $p\in L^2 ([0,T]; H^1 )$ such that the vector field $u(t)$ satisfies in the weak ($L^2$) sense the equation

\begin{eqnarray*}
\frac{\partial}{\partial t}(u-\Delta u)-\nu\Delta(u-\Delta u)=-u\cdot \nabla(u-\Delta u)-\nabla p
\end{eqnarray*}
with $\div u(t,\cdot )=0 ~\forall t$.

\rm
\vskip 3mm

\begin{proof} The proof is similar to the one of Theorem 3.1.
We have
\begin{eqnarray*}
\frac{d}{d\varepsilon}|_{\varepsilon=0}A[g^{\varepsilon}]&=&
\sum_{i}E\int\int<\frac{d}{d\varepsilon}|_{\varepsilon=0}\{D(g_t^{\varepsilon})^i\circ (g_t^{\varepsilon})^{-1}\},u^i(t,\theta)>dtd\theta\\
&&+\sum_{i,j}E\int\int<\frac{d}{d\varepsilon}|_{\varepsilon=0}\partial_{j}\{D(g_t^{\varepsilon})^i\circ (g_t^{\varepsilon})^{-1}\},\partial_{j}u^i(t,\theta)>dtd\theta\\
&\triangleq&A_1+A_2.
\end{eqnarray*}
Since 
\begin{eqnarray*}
D(g_t^{\varepsilon})^i=[\partial_t (e_t ^{\varepsilon v})^i+(u\cdot\nabla) (e_t^{\varepsilon v})^i+\nu\Delta (e_t^{\varepsilon v})^i](g_t^{\varepsilon})
\end{eqnarray*}
and
\begin{eqnarray*}
D(g_t^{\varepsilon})^i\circ (g_t^{\varepsilon})^{-1}=\partial_t (e_t ^{\varepsilon v})^i+(u\cdot\nabla) (e_t^{\varepsilon v})^i+\nu\Delta (e_t^{\varepsilon v})^i,
\end{eqnarray*}
we have
\begin{eqnarray*}
&&\frac{d}{d\varepsilon}|_{\varepsilon=0}\{D(g_t^{\varepsilon})^i\circ (g_t^{\varepsilon})^{-1}\}\\
&=&\frac{d}{d\varepsilon}|_{\varepsilon=0}\{\partial_t (e_t ^{\varepsilon v})^i+(u\cdot\nabla) (e_t^{\varepsilon v})^i+\nu\Delta (e_t^{\varepsilon v})^i\}\\
&=&\partial_t v^i(t,\theta)+\sum_l u^l(t,\theta)\partial_lv^i(t,\theta)+\nu\Delta v^i(t,\theta)
\end{eqnarray*}
and
\begin{eqnarray*}
&&\frac{d}{d\varepsilon}|_{\varepsilon=0}\partial_j\{D(g_t^{\varepsilon})^i\circ (g_t^{\varepsilon})^{-1}\}\\
&=&\partial_j\frac{d}{d\varepsilon}|_{\varepsilon=0}\{D(g_t^{\varepsilon})^i\circ (g_t^{\varepsilon})^{-1}\}\\
&=&\partial_j\{\partial_t v^i(t,\theta)+\sum_l u^l(t,\theta)\partial_lv^i(t,\theta)+\nu\Delta v^i(t,\theta)\}\\
&=&\partial_t\partial_j v^i+\partial_j[\sum_l u^l(t,\theta)\partial_lv^i(t,\theta)]+\nu\Delta\partial_jv^i.
\end{eqnarray*}
Then
\begin{eqnarray*}
A_1&=&
\sum_{i}E\int\int<\frac{d}{d\varepsilon}|_{\varepsilon=0}\{D(g_t^{\varepsilon})^i\circ (g_t^{\varepsilon})^{-1}\},u^i(t,\theta)>dtd\theta\\
&=&\sum_{i}E\int\int<\partial_t v^i(t,\theta)+\sum_l u^l(t,\theta)\partial_lv^i(t,\theta)+\nu\Delta v^i(t,\theta)
,u^i(t,\theta)>dtd\theta\\
&=&\sum_i[-\int\int\partial_tu^iv^idtd\theta+\nu\int\int\Delta u^iv^idtd\theta
-\int\int(u\cdot\nabla)u^iv^idtd\theta ]\\
&=&-\int\int<\partial_tu,v>dtd\theta+\int\int<\nu\Delta u,v>dtd\theta
-\int\int<(u\cdot\nabla)u,v>dtd\theta,
\end{eqnarray*}
\begin{eqnarray*}
A_2&=& \sum_{i,j}E\int\int<\frac{d}{d\varepsilon}|_{\varepsilon=0}\partial_{j}\{D(g_t^{\varepsilon})^i\circ (g_t^{\varepsilon})^{-1}\},\partial_{j}u^i>dtd\theta\\
&=&\sum_{i,j}E\int\int<\partial_t\partial_j v^i+\partial_j[\sum_l u^l(t,\theta)\partial_lv^i(t,\theta)]+\nu\Delta\partial_jv^i,\partial_j u^i>dtd\theta\\
&=&\sum_{i,j}[\int\int\partial_t\Delta u^i v^idtd\theta-\nu\int\int\Delta^2u^i v^i dtd\theta
-\int\int \sum_l u^l(t,\theta)\partial_lv^i(t,\theta) \Delta u^i(t,\theta)  dtd\theta]\\
&=&\sum_{i,j}[\int\int\partial_t\Delta u^i v^idtd\theta-\nu\int\int\Delta^2u^i v^i dtd\theta
+\int\int \sum_l v^i(t,\theta) u^l(t,\theta)\partial_l\Delta u^i(t,\theta)  dtd\theta]\\
&=&\int\int<\partial_t\Delta u, v>dtd\theta-\int\int<\nu\Delta^2u, v> dtd\theta
+\int\int<(u\cdot\nabla)\Delta u, v>dtd\theta,
\end{eqnarray*}

So $\frac{d}{d\varepsilon}|_{\varepsilon=0}A[g^{\varepsilon}]=0$ is equivalent to
\begin{eqnarray*}
\partial_t(u-\Delta u)-\nu\Delta(u-\Delta u)=-u\cdot \nabla(u-\Delta u)-\nabla p
\end{eqnarray*} 
with $\div u(t,\cdot )=0$ for all $t\in [0,T]$. The equation holds in the weak sense.
\end{proof}

As the admissible variations preserve the class ${\mathcal S}_0$, one can show by the same methods used in Theorem 4.1,
the following
\vskip 3mm
\bf Theorem 5.2. \it 
 Let $z \in L^2 ([0,T];L^2 )$ be a vector field and $c$ a positive constant.
There exists a semimartingale $g(t)$ in the class ${\mathcal S}_0$
which realizes the minimum of the action functional $A$.
Then the corresponding drift
$u(t, \cdot )$ satisfies the incompressible Leray-alpha equation equation in the weak $L^2$ sense. Moreover
$\int \int <u(t,\theta ), z(t,\theta > dt d\theta \geq c$.
\rm
\vskip 5mm

\bf Acknowledgements: \rm  ~ 
The first author thanks Prof. D.D. Holm for guiding her into the literature on the PDE's  in turbulence theory.

The second author  is supported by State Scholarship Fund of China and acknowledges the Centre Interfacultaire Bernoulli, EPFL, Switzerland, for the invitation to visit the center during two weeks. Both authors have been partly supported by the project PTDC/MAT-CAL/0749/2012, FCT, Portugal.

%%%%%%%%%%%%%%%%%%%%%%%%%%%%%%%%%%%%%%%%%%%%%%%%%%%%%%%%%%%%%%%%%%%%%%%%%
%
%   R E F E R E N C E S 
%
%%%%%%%%%%%%%%%%%%%%%%%%%%%%%%%%%%%%%%%%%%%%%%%%%%%%%%%%%%%%%%%%%%%%%%%%%

\providecommand{\bysame}{\leavevmode\hbox to3em{\hrulefill}\thinspace}


\begin{thebibliography}{10}

\bibitem{Arnaudon-Chen-Cruzeiro:12} M.~Arnaudon, X.~Chen and A.B.~Cruzeiro, \textit{Stochastic Euler-Poincar\'e reduction}, J. Math. Physics 55 (2014), 081507



\bibitem{Arnaudon-Cruzeiro:12} M.~Arnaudon and A.B.~Cruzeiro, \textit{Lagrangian Navier-Stokes diffusions on manifolds: variational principle and stability}, Bull. Sci. Math., 136, 8 (2012), 857⠀"-881.




  \bibitem{Arnold:66} V. I. ~Arnold, \textit{Sur la g\'eom\'etrie diff\'erentielle des groupes de Lie de dimension infinie et ses applications \`a l'hydrodynamique des fluides parfaits}, Ann. Inst. Fourier  16  (1966),  316--361.

\bibitem{BS:2008} C.~ Bjorland and M. E.~ Schonbek, \textit{On questions of decay and existence for the viscous Camassa --Holm
equations}, Ann. I. H. Poincar\'e Ð AN 25 (2008), 907--936,
  
  \bibitem{Camassa-Holm:93} R.~Camassa and D.D.~Holm, \textit{An integrable shallow water equation with peaked solitons}, Phys. Rev. Lett.  71  (1993),  n. 1, 1661--1664.
  
  \bibitem{Cheskidov-etal:05} A.~Cheskidov, D.D.~Holm, E.~Olson and E.S.~Titi, \textit{On a Leray-$\alpha$ model of turbulence}, Proc. Royal Soc. A
  461 (2005), 629--649.
  \bibitem{Cipriano-Cruzeiro:07} F.~Cipriano and A.B.~Cruzeiro, \textit{Navier-Stokes equation and diffusions on the group of homeomorphisms of the torus}, Comm. Math. Phys.  275  (2007),  n. 1, 255--269.
  
  
  \bibitem{Constantin-Escher:98} A.~Constantin and J.~Escher, \textit{Well-Posedness, Global Existence,
and Blow-up Phenomena for a Periodic Quasi-Linear Hyperbolic Equation}, Comm. Pure Appl. Math.  51 (5)  (1998) 475--504.

\bibitem{Chen-Cruzeiro-Ratiu:15} X. ~Chen, A.B. ~Cruzeiro and T.~Ratiu, \textit{Constrained and stochastic variational principles for dissipative equations
 with advected quantities},    arXiv:1506.05024 (2015).
 
 \bibitem{Chen-etal:98} S.~Chen, C.~Foias, D.D.~Holm, E.~Olson, E.S.~Titi and S. ~ Wynne, \textit{The CamassaÐHolm
equations as a closure model for turbulent channel and pipe flow}, Phys. Rev. Lett. 81 (1998), 5338--5341.
 
   \bibitem{Chen-etal:99} S.~Chen, C.~Foias, D.D.~Holm, E.~Olson, E.S.~Titi and S. ~ Wynne, \textit{The CamassaÐHolm
equations and turbulence.}, Physica D 133 (1999), 66--83.
   
 \bibitem{Ebin-Marsden:70} D. G. ~Ebin and J. E. ~Marsden, \textit{Groups of diffeomorphisms and the motion
  of an incompressible fluid}, Ann. of  Math. (2) 92 (1970), 102--163.


\bibitem{Fang:02} S. ~Fang, \textit{Canonical Brownian motion on the diffeomorphism group of the circle}, J. Funct. Anal. 196 (2002), 162--179.


 \bibitem{Fang:04} S. ~Fang, \textit{Solving stochastic differential equations on Homeo ($S^1$)}, J. Funct. Anal. 216 (2004), 22--46.
 
 \bibitem{Fang-Li-Luo:11} S.~Fang, H.~Li and D.~Luo, \textit{Heat semi-group and generalized flows on complete
Riemannian manifolds}, Bull. Sci. Math., 135 (2011), 565--600.
 
 \bibitem{Foias-etal:02} C.~Foias, D.D.~Holm and E.S.~Titi \textit{The three-dimensional viscous CamassaÐHolm equations,
and their relation to the NavierÐStokes equations and turbulence theory}, J. Dyn. Diff. Eq, 14  (2002), 1--35.
 
 \bibitem{Foias-etal:01} C. ~Foias, D.D.~Holm and E.S.~Titi \textit{The Navier-Stokes-alpha model of fluid turbulence}, Physica D 152--153 (2001), 505--519.
 
  \bibitem{Holm-Marsden:04} D. D. ~Holm and J. E.~Marsden, \textit{Momentum maps and measure-valued solutions (peakons, filaments and sheets) for the EPDiff equation}, Progr. Math., 232, J. E. Marsden and T. S. Ratiu, Editors, Birkhauser Boston  (2004).
  
  


\bibitem{Ikeda-Watanabe:81} N. ~Ikeda and S.~Watanabe, \textit{Stochastic differential equations and diffusion processes}, North-Holland, Amsterdam, 1981.
   
\bibitem{Kouranbaeva:99}  S.~Kouranbaeva, \textit{The Camassa-Holm equation as a geodesic flow on the diffeomorphisms group}, J. Math. Phys. 40 (1999), 857--868.
 

  
\bibitem{Kunita:90} H. ~Kunita, \textit{Stochastic flows and stochastic differential equations},
Cambridge Univ. Press, 1990.

\bibitem{Lim:07} W. K. ~Lim, \textit{Global well-posedness for the viscous Camassa --Holm equation}, J. Math. Anal. Appl. 326 (2007) , n.1, 432--442.

 \bibitem{Marsden-Ratiu:02}  J. E. ~Marsden and T.~Ratiu, \textit{Introduction to Mechanics and Symmetry: A Basic Exposition of Classical Mechanical
Systems}, Texts in Applied Mathematics, Springer 2002.


\bibitem{Malliavin:99}  P.~Malliavin, \textit{The canonic diffusion above the diffeomorphism group of the circle}, C. R. Acad. Sci. Paris 329 (1999), 325--329.


\bibitem{Misiolek:98}  G.~Misiolek, \textit{A shallow water equation as a geodesic flow on the Bott-Virasoro group}, J. Geom. Phys. 24 (1998),  203--208.

\bibitem{Shkoller:98}  S.~Shkoller, \textit{Geometry and Curvature of Diffeomorphism Groups with $H^1$ Metric and Mean Hydrodynamics}, J. Funct. Anal. 160 (1998), 337--365.

\bibitem{Tan-Shen-Ding:07} L.~Tian, C.~Shen and D.~Ding, \textit{Optimal control of the viscous Camassa --Holm equation}, Nonlinear Anal.: Real World Appl. 10 (2009), 519--530.




\end{thebibliography}
\end{document}